\documentclass[10pt]{amsart}

\usepackage[british]{babel}
\usepackage{amsthm,amsmath,amsfonts,amssymb,mathrsfs,mathtools}
\usepackage[colorlinks,citecolor=blue,urlcolor=blue]{hyperref}
\usepackage{hhline,multirow}
\usepackage[table,xcdraw]{xcolor}
\usepackage{texdraw,extarrows}
\usepackage{url}
\usepackage{amstext}
\usepackage{amsopn}
\usepackage{pgfplots,geometry,subfigure}
\usepackage{color}
\usepackage[justification=justified]{caption}
\usepackage{mhchem,gensymb}
\usepackage{tikz,graphicx}
\usepackage{bbm,pifont}

\newcommand{\sK}{{\sf K}}
\newcommand{\eS}{\emph{S}}
\newcommand{\tS}{\text{S}}

\newcommand{\lt}{\left}
\newcommand{\rt}{\right}
\newcommand{\la}{\langle}
\newcommand{\ra}{\rangle}
\newcommand{\lf}{\lfloor}
\newcommand{\rf}{\rfloor}
\newcommand{\om}{\omega}
\newcommand{\N}{\mathbb{N}}

\newcommand{\R}{\mathbb{R}}
\newcommand{\Z}{\mathbb{Z}}

\newcommand{\cF}{\mathcal{F}}
\newcommand{\cI}{\mathcal{I}}
\newcommand{\cO}{\mathcal{O}}

\newcommand{\cR}{\mathcal{R}}
\newcommand{\cS}{\mathcal{S}}
\newcommand{\cK}{\mathcal{K}}

\newcommand{\cT}{\Omega}
\newcommand{\cC}{\mathfrak{C}}
\newcommand{\cX}{\mathcal{Y}}
\newcommand{\rO}{\mathrm{O}}
\newcommand{\sE}{{\sf E}}
\newcommand{\si}{{\sf i}}
\newcommand{\sL}{{\sf L}}
\newcommand{\sN}{{\sf N}}

\newcommand{\so}{{\sf o}}
\newcommand{\sP}{{\sf P}}
\newcommand{\sQ}{{\sf Q}}

\newcommand{\sS}{{\sf S}}
\newcommand{\sT}{{\sf T}}
\newcommand{\supp}{{\rm supp\,}}

\newcommand{\uc}{c_*}
\newcommand{\bc}{c^*}

\usetikzlibrary{arrows, decorations.markings}
\usetikzlibrary{positioning}
\usetikzlibrary{fit}
\usetikzlibrary{backgrounds}
\usetikzlibrary{calc}
\usetikzlibrary{shapes}
\usetikzlibrary{mindmap}
\usetikzlibrary{decorations.text}
\pgfplotsset{compat=newest}
\numberwithin{equation}{section}
\theoremstyle{plain}
\newtheorem{theorem}{Theorem}[section]

\newtheorem{corollary}[theorem]{Corollary}
\newtheorem{proposition}[theorem]{Proposition}

\theoremstyle{definition}
\newtheorem{definition}[theorem]{Definition}
\newtheorem{example}[theorem]{Example}
\theoremstyle{remark}
\newtheorem{rem}[theorem]{Remark}
\makeatletter%

\newcommand\xldownrupharpoon[2][]{%
\ext@arrow 0099{\ldownrupharpoonfill@}{#1}{#2}}
 \def\ldownrupharpoonfill@{%
\arrowfill@\leftharpoondown\relbar\rightharpoonup}
\makeatother
\newcommand{\srcsize}{\@setfontsize{\srcsize}{5pt}{5pt}}

\makeatletter
\@namedef{subjclassname@2020}{\textup{2020} Mathematics Subject Classification}
\makeatother

\begin{document}
\title[Classification and threshold dynamics of SRNs]
  {Classification and threshold dynamics of stochastic reaction networks}

\author{Carsten Wiuf}
\address{
Department of Mathematical Sciences,
University of Copenhagen, Copenhagen,
2100 Denmark.}
\email{wiuf@math.ku.dk}
\author{Chuang Xu}
\address{
Department of Mathematics, University of Hawai'i at M\={a}noa, Honolulu, 96822, HI, USA.}
\email{chuangxu@hawaii.edu}


\date{\today}

\noindent

\begin{abstract}
 \emph{Stochastic reaction networks} (SRNs) provide models of many real-world networks. Examples include networks  in epidemiology, pharmacology, genetics, ecology, chemistry, and social sciences. Here, we model stochastic reaction networks  by continuous time Markov chains (CTMCs) and 
pay special attention to one-dimensional mass-action SRNs (1-d stoichiometric subspace). We classify all states of the underlying CTMC of 1-d SRNs. In terms of (up to) four parameters, we provide \emph{sharp checkable criteria} for various dynamical properties (including explosivity, recurrence, ergodicity, and the tail asymptotics of stationary or quasi-stationary distributions) of SRNs in the sense of their underlying CTMCs. As a result, we prove that all 1-d endotactic networks are non-explosive, and positive recurrent with an ergodic stationary distribution with Conley-Maxwell-Poisson (CMP)-like tail, provided the state space of the associated CTMCs consists of closed communicating classes. In particular, we prove the recently proposed \emph{positive recurrence conjecture} in one dimension: Weakly reversible mass-action SRNs with 1-d stoichiometric subspaces are positive recurrent. The proofs of the main results rely on our recent work on CTMCs with polynomial transition rate functions.\end{abstract}

\maketitle

\section{Introduction}

Many real-world phenomena can be modelled as reaction networks.  Examples include networks in epidemiology
 \cite{B09}, pharmacology  \cite{BI09}, ecology \cite{DDK19}, and social sciences \cite{P18}, as well as
 gene regulatory networks \cite{C18},  chemical reaction networks \cite{F19}, signalling networks \cite{PHPS05},   and metabolic
networks \cite{SZHS12}.  For example, the number of  individuals infected with some contagion in a population might be modelled by $S+I\to 2I$, $I\to R$, where $S$ denotes a susceptible individual, $I$ an infected and $R$ a recovered individual.

Noise is ubiquitous. The dynamics of the species composition in a reaction network might therefore be modelled as a  continuous time Markov chains (CTMCs) with an intensity for each reaction to occur. This applies in particular when the species counts  are low  \cite{AK11}. Alternatively, it might be modelled deterministically  in terms of ordinary differential equations (ODEs), when  the species abundances are high and noise is averaged out \cite{F19}.

This paper aims to contribute to the theory of {\em stochastic reaction networks} (SRNs), especially in dimension one. In our context, CTMCs are associated with an SRN on the ambient space $\N^d_0$, where $d$ is the number of species. To understand the dynamics of such CTMCs, it is  important to understand the decomposition of $\N^d_0$ into communicating classes.
Subsequently, one might decide whether a CTMC is recurrent, transient, explosive, or extinct on a particular class. In general this decomposition seems difficult to achieve, and even in concrete cases, it might be hard to disentangle open classes from  closed classes. When the stoichiometric subspace of the SRN is one-dimensional, we provide an explicit full classification for  all sorts of communicating classes (Theorem~\ref{th-1d-classification}).

We next turn towards dynamical properties (e.g., explosivity, recurrence, ergodicity, quasi-ergodicity) of the CTMCs associated with an SRN. This has been a topic of recent interest. For example,   \emph{complex-balanced} mass-action SRNs have been shown to be non-explosive, using novel sufficient conditions  for non-explosivity of CTMCs  \cite{ACKK18}. Indeed, the mechanism for SRNs to be explosive (or non-explosive) is largely unknown \cite{ACKK18}. To help fill this gap, we provide a necessary and sufficient condition for non-explosivity of mass-action SRNs with one-dimensional  stoichiometric subspace. This is accomplished in terms of four parameters, that provide threshold criteria for (non-)explosivity of SRNs (Theorem~\ref{th-Threshold}).

Another topic of prime importance
is \emph{ergodicity}, which generally guarantees the well-posedness of a SRN model. A scalable computational sufficient condition for ergodicity of essential SRNs has been proposed, and  applied  to gene regulatory networks \cite{GBK14}.
Similar results have recently been established for ergodicity of quasi-stationary distributions (QSDs) for extinct SRNs  \cite{HW18}. Again, for mass-action SRNs with one-dimensional  stoichiometric subspace, we provide a sharp criterion for ergodicity as well as quasi-ergodicity (Theorem~\ref{th-Threshold}), also in terms of the four aforementioned parameters.

Complex-balanced SRNs are  weakly reversible and  positive recurrent with an ergodic Poisson product-form stationary distribution \cite{ACK10,ACKK18}.
For the extended class of weakly reversible reaction networks, positive recurrence as well as exponential ergodicity has been showcased in some cases \cite{AK18,ACK19,AK22}.
It is  conjectured  that all weakly reversible mass-action SRNs are positive recurrent, as an analogue of the celebrated Global Attractor Conjecture for deterministic reaction networks \cite{AK18}. 
We verify this positive recurrence conjecture in the affirmative for all one-dimensional mass-action SRNs (Corollary~\ref{cor-posrec}).

Since weakly reversible mass-action SRNs are essential and  endotactic, we indeed prove a stronger result: in dimension one, all mass-action essential endotactic SRNs are positive recurrent (Theorem~\ref{th-endotactic}). Furthermore, the stationary distributions have tails like the \emph{Conley-Maxwell-Poisson} (CMP) distribution, a generalization of the Poisson distribution known from complex-balanced SRNs (Corollary~\ref{cor-end-tail}). Indeed, we provide a trichotomy for one-dimensional ergodic SRNs regarding the tails of their stationary distributions as well as QSDs (Theorem~\ref{th-tail}): The tails decay  like CMP distributions (super-exponential), geometric distributions (exponential), or power-law distributions (sub-exponential).

\section{Preliminaries}
\subsection{Notation}\label{sec:notation}

Let $\R$, $\R_{\ge 0}$, $\R_{>0}$ be the set of real, non-negative real, and positive real  numbers, respectively. Let $\Z$ be the set of integers, $\N=\Z\cap \R_{>0}$ and $\N_0=\N\cup\{0\}$.

For $a\in\R$, let $\lceil a\rceil$ be the ceiling function (i.e., the minimal integer $\ge a$), and $\lfloor a\rfloor$ the floor function (i.e., the maximal integer $\le a$).
 For $x=(x_1,\ldots,x_d)\in\R^d$,  
 define:
$$\widehat x_j=(x_1,\cdots,x_{j-1},x_{j+1},\cdots,x_d).$$
For $x,y\in\R^d$, let $\|x\|_1$ denote the $\ell_1$-norm of $x$, $\la x,y\ra$  the inner product of $x$ and $y$, and $x\ge y$ (similarly, $y\le x$, $x>y$, $y<x$)  if it holds coordinate-wise. Furthermore,  
let $x\Z+y=\{z\in\R^d|z=xn+y, n\in\Z\}$.  
 For a non-empty subset $A\subseteq\N^d_0$, let
$\text{cl}(A)=\{x\in \N^d_0\colon x\ge y\,\, \text{for some}\ y\in A\}$ be the closure of $A$. 
A set $A$ is the {\em minimal set} of a 
set $B\subseteq\N_0^d$ if $A$ consists of all elements $x\in B$ such that $x\not\ge y$ for all $y\in B\setminus\{x\}.$ In particular, $A$ is {\em finite} and $\text{cl}(A)=\text{cl}(B)$ \cite{XHW20a}. 

Let $f$ and $g$ be non-negative functions defined on an unbounded set $A\subseteq\R_{\ge 0}$. We denote $f(x)\lesssim g(x)$ if there exists $C,N>0$ such that $f(x)\le C g(x)$ for all  $x\in A$, $x\ge N,$
that is, $f(x)=\rO(g(x))$, since $f$ is non-negative. Here, $\rO$ refers to the standard big O notation.

\subsubsection{Greatest common divisor} \label{sec:gcd}

For $x,y\in\Z^d$,  $x\neq(0,\cdots,0)$, we say $x$ is a {\em (positive) divisor} of $y$, if there exists  $a\in\Z$ ($a\in\N$), such that $y=ax$.  For $A\subseteq\Z^d$, $x$ is a {\em common (positive) divisor} of $A$, 
 if it is a common (positive) divisor of all $y\in A$. Moreover, $x$ is the {\em greatest common divisor} (gcd) of $A$, denoted $\gcd(A)$, if $x$ is a common divisor of all other common divisors of $A$, and the first non-zero coordinate of $x$ is positive. Hence, $\gcd(A)$ is unique, if it exists.
Not all subsets of $\Z^d$ have a gcd (or even a common divisor), e.g., $A=\{(1,2),(2,1)\}$. Indeed, a non-empty set  $A\subseteq\Z^d\setminus\{(0,\cdots,0)\}$ is contained in a line,  if and only if $A$ has a common divisor, if and only if $\gcd(A)$ exists \cite{XHW20a}.

For $x\in\Z^d$, let $\gcd(x)=\gcd(\{x_j\colon j=1,\ldots,d \})\in\N$.

\subsubsection{Lattice interval} \label{sec:lattice}

For $x,y\in\N^d_0$, denote by $[x,y]_1$ the lattice line segment  
in $\N_0^d$ from $x$ to $y$, referred to as the (closed) {\em lattice interval} between $x$ and $y$.  
Similarly, $[x,y[_1=[x,y]_1\setminus\{y\}$,  For $d=1$, $[x,y]_1$ is the set of integers from $x$ to $y$.   
and analogously for $]x,y[_1$ and $]x,y]_1$.
 Moreover, for a non-empty subset $A\subseteq\N^d_0$ on a {\em line}, if $A$ is not perpendicular to the axis of the first coordinate, then the elements in $A$ are comparable with respect to the  order induced by the first coordinate, denoted $\le_1$.
 Let $\min_1 A$ and $\max_1 A$ be the unique minimum and maximum induced by the order.  

\subsection{Markov chains}\label{sec:mc}

Let  $Q$ be a conservative $Q$-matrix on $\N^d_0$ with entries $q_{x,x'},$ $x, x'\in\N_0^d$ \cite{R57}. Let
$\cT=\{x'-x\colon q_{x,x'}>0\ \text{for some}\ x, x'\in\N^d_0\}$
 be the  set of \emph{jump vectors}, and define the \emph{transition rate functions} by
\[\lambda_{\om}\colon \N^d_0\to[0,+\infty), \quad \lambda_\om(x)= q_{x,x+\om},\quad  x\in\N_0^d, \quad \om\in\cT.\]

 A state $x'\in\N_0^d$ is {\em one-step reachable} from $x\in\N^d_0$  if $q_{x,x'}>0$.
  A state $x'\in\N_0^d$ is {\em reachable} from $x\in\N^d_0$ (or equivalently, $x$ \emph{leads to} $x'$), denoted by $x\rightharpoonup y$,
if there exists a sequence of states $x^{(j)}$, $j=1,\ldots,m$ for some $m\in\N_0$, such that $x=x^{(1)}$ and $x'=x^{(m)}$, and $x^{(j)}$ is one-step reachable from $x^{(j-1)}$, $j=2,\ldots,m$.  
In particular, $x$ is reachable from itself.
A state $x$ \emph{communicates} with $x'$,  if  $x\rightharpoonup x'$ and $x'\rightharpoonup x$.  
Hence, `communicates'   
defines an equivalence relation on $\N^d_0$, and partitions $\N^d_0$ into {\em communicating classes}.  
A non-empty subset $E\subseteq\N^d_0$ is {\em closed} if $x\in E\ \text{and}\ x\rightharpoonup x'\ \text{implies}\ x'\in E$. A set is {\em open} if it is not closed. A state $x$ is {\em absorbing} ({\em escaping}) if $\{x\}$ is a closed (open) class. An absorbing state $x$ is {\em neutral} if $x$ is not reachable from any other state $x'\not=x$,  
and otherwise, it is {\em trapping}. A non-singleton communicating class is an {\em irreducible component} (IC). A closed IC is a {\em positive irreducible component} (PIC), while an open IC is a {\em quasi-irreducible component} (QIC). Any singleton communicating class is either an absorbing state (neutral or trapping), or an escaping state.  
Let $\sN$, $\sT$, $\sE$, $\sP$, and $\sQ$ be the (possibly empty) set of neutral states, trapping states, escaping states, positive irreducible components and quasi-irreducible components for $Q$, respectively.

A CTMC $(Y_t)_{t\ge 0}$   with initial state $Y_0=x$ is \emph{explosive} at $x$ if it jumps infinitely many times in finite time with positive probability, and is \emph{non-explosive} at $x$ otherwise.  
Furthermore,  $(Y_t)_{t\ge 0}$   is {\em recurrent} ({\em positive} or {\em null recurrent}), or {\em transient} on a PIC according to the standard meanings of the terms \cite{N98}.

Let $\tau_{\partial}=\inf\{t>0\colon Y_t\in\partial\}$ be the entrance time of $(Y_t)_{t\ge 0}$ into a set of states $\partial\subseteq \cX$. We say, $(Y_t)_{t\ge 0}$ has \emph{certain absorption} into $\partial$ if $\tau_{\partial}<\infty$ almost surely for all $Y_0\in\partial^{\sf c}=\cX\setminus\partial$. Moreover, the process associated with $(Y_t)_{t\ge 0}$, conditioned to  be never absorbed, is called a {\em $Q$-process} \cite{CV16}.
A probability measure $\nu$ on $\partial^{\sf c}$ is a quasi-stationary distribution (QSD) of $(Y_t)_{t\ge 0}$ if for all $t\ge0$,
\[\mathbb{P}_{\nu}(Y_t\in A|\tau_{\partial}>t)=\nu(A),\quad  A \subseteq \partial^{\sf c},\]
where $\mathbb{P}_{\nu}$ is the conditional probability measure such that  $Y_0\sim\nu$.

A probability measure $\pi$ on $\N^d_0$ is a \emph{stationary distribution} if $\pi$ is a  non-negative equilibrium of the so-called {\em master equation} \cite{G92}:
 \[0=\sum_{\omega\in\cT}\lambda_{\omega}(x-\omega)\pi(x-\omega)-\sum_{\omega\in\cT}\lambda_{\omega}(x)\pi(x),\quad x\in\N_0^d,\]
where $\lambda_\omega(x)$ is set to zero if $x\not\in\N_0^d$.

\subsection{Stochastic reaction networks}

A \emph{reaction network}  is  a triplet of finite nonempty sets $(\cS, \mathcal{C}, \cR)$, where:
  \begin{enumerate}
    \item[$\bullet$] $\cS=\{\tS_j\}_{j=1}^d$ is a set of symbols, termed {\em species};
    \item[$\bullet$] $\mathcal{C}$ is a set of linear combinations of species on $\N_0$, termed {\em complexes}; and
    \item[$\bullet$] $\cR\subseteq \mathcal{C}\times\mathcal{C}$ is a set of \emph{reactions}, such that $(y,y)\not\in\cR$.  A reaction $(y,y')$ is denoted $y\to y'$. The complex $y$ is called the {\em reactant} and $y'$ the {\em product}.
  \end{enumerate}
  
  We assume every species is in some complex and every complex is in some reaction. Hence, we identify $\cR$ with the reaction network  as $\cS$ and $\mathcal{C}$  can be found from $\cR$.
  The species are ordered such that $\cS$ is identified with $\{1,\ldots,d\}$ and the complexes are identified with vectors in $\N^d_0$, for instance, for $d=3$, the complex $\tS_1+2\tS_3$ is expressed as $(1,0,2)$.   The pair $(\mathcal{C},\cR)$ forms a digraph in $\N^d_0$ referred to as the {\em reaction graph}.

 A \emph{stochastic reaction network} (SRN) is a pair $(\cR,\cK)$, consisting of a reaction network  $\cR$ and a \emph{kinetics} $\mathcal{K}=(\eta_{y\to y'})_{y\to y'\in\cR}$, where  $\eta_{y\to y'}\colon\N^d_0\to\R_{\ge0}$ is the {\em rate function} of $y\to y'$, expressing the propensity of the reaction  to occur.     A special kinetics is {\em stochastic mass-action kinetics},
\[\eta_{y\to y'}(x)=\kappa_{y\to y'}\prod_{i=1}^d x_i(x_i-1)\ldots(x_i-y_i+1)=\kappa_{y\to y'}x^{\underline{y}},\quad x=(x_1,\ldots,x_d),\]
where the latter equality is a definition.
Hence,  for $x\in\N^d_0$, $\eta_{y\to y'}(x)$ is positive  if and only if $x\ge y$.

 For an SRN $(\cR,\cK)$, define a $Q$-matrix on $\N^d_0$ by  
   \[q_{x,x'}=\sum_{y\to y'\in\cR\colon y'-y=x'-x}\eta_{y\to y'}(x),\quad x,x'\in\N^d_0,\]
  The set $\cT=\{y'-y\colon\ y\to y'\in\cR\}$ is   the set of jump vectors of $Q$, called the set of  \emph{reaction vectors} in the present context. The $Q$-matrix defines a decomposition of $\N^d_0$ into communicating classes and into the sets $\sN$, $\sT$, $\sE$, $\sP$, and $\sQ$, see Section \ref{sec:mc}.

Finally, we recall some important types of reaction networks well-studied in the literature \cite{CNP13,GMS14}. A reaction network is {\em weakly reversible} if
the reaction graph is a finite  union of strongly connected components. A larger class of reaction networks, which includes weakly reversible reaction networks \cite{GMS14}, is \emph{endotactic networks} (see Appendix~\ref{appendix_endotactic_network}  for definition and some properties). A subclass  is \emph{strongly endotactic networks}. Strongly endotactic deterministic mass-action   reaction networks are permanent \cite{GMS14}. In contrast,  strongly endotactic stochastic reaction networks  with   mass-action kinetics may be transient or explosive \cite{ACKN18}.

\section{Structural classification of SRNs}

In this section, we investigate the relationship between  the structure of the reaction graph of an SRN and the structure of the corresponding $Q$-matrix.
This relationship does not rely on the specific kinetics  provided the  kinetics fulfils a mild condition. Specifically, for an SRN $(\cR,\cK)$ with $\cK=(\eta_{y\to y'})_{y\to y'\in\cR}$,   we assume the following  condition:

\medskip
($\rm\mathbf{H1}$) For  $y\to y'\in\cR$, $\eta_{y\to y'}(x)>0$ if and only if $x\ge y$, $x\in\N^d_0$.
\medskip

In particular, stochastic mass-action fulfils 
($\rm\mathbf{H1}$). Furthermore, the assumption ensures the ambient space $\N^d_0$ is invariant \cite{AK15}, and that the partitioning of $\N^d_0$ into communicating classes is independent of the kinetics, subject to ($\rm\mathbf{H1}$).

\medskip
($\rm\mathbf{H2}$) $\dim\sS=1$.
\medskip

 From ($\rm\mathbf{H2}$), the greatest common divisor of $\cT$ exists, $\om^*=(\om^*_1,\ldots,\om^*_d)=\gcd(\cT)$, see Section \ref{sec:gcd}. Note that $\om^*$ may not be in $\cT$. 
 Let $\om^{**}=\gcd(\{\om^*_1,\ldots,\om^*_d\})$.

Let $\cR_{\pm}=\{y\to y'\in\cR\colon \pm(y'-y)\cdot\om^*>0\}$ be the sets of {\em positive} and {\em negative reactions}, respectively. Then $\cR=\cR_+\cup\cR_-$ provides a decomposition of $\cR$ into two subnetworks.
 To avoid trivial dynamics (like pure-birth or pure-death processes), we assume

\medskip
 ($\rm\mathbf{H3}$) $\cR_-\neq\varnothing$ and $\cR_+\neq\varnothing$.
\medskip

\medskip
 ($\rm\mathbf{H4}$) $\om^*\in\N^d$.

\medskip
There are no infinite ICs when ($\rm\mathbf{H4}$) fails, and hence the dynamics can be classified by known means   \cite{CMM13}.
We are interested in the dynamics of CTMCs on infinite state spaces associated with an SRN, assuming $\rm{(\mathbf{H1})}$-$\rm{(\mathbf{H4})}$. 
For regularity, we impose a simplifying assumption:

\medskip
($\rm\mathbf{H5}$) For each species in $\cS$, we assume  its stoichiometric coefficient changes in at least one reaction.  
\medskip

By ($\rm\mathbf{H2}$) and ($\rm\mathbf{H5}$),  $\om^*_j\neq0$ for all $j=1,\ldots,d$. If there are species not fulfilling assumption ($\rm\mathbf{H5}$), then their numbers are constant in time (as the stoichiometric coefficients are unchanged). These numbers might be `absorbed' inot the reaction rate constants, thereby allowing the results stated below to be applied even if  ($\rm\mathbf{H5}$) fails.

\begin{definition}\label{def:ess-ext}
 $(\cR,\cK)$ is \emph{essential} if  $\N^d_0=\sN\cup\sP$, and \emph{extinct} if  $\sP=\varnothing$.
\end{definition}

 Consequently, an SRN can be  essential, extinct, or  neither essential nor extinct. For an essential SRN, its ambient space consists of closed communicating classes. In contrast, for an extinct SRN, all   states are transient or neutral. When an SRN is  neither essential nor extinct, the ambient space $\N^d_0$  contains both PICs and open communicating classes (i.e., escaping states or QICs). 

Let $\cI$ (inputs) and $\cO$ (outputs) be the set of reactants and products, respectively. For  $\omega\in\cT$, let $\cI_{\om}$ be the minimal set of the reactants $\{y\colon y\to y'\in\cR, y'-y=\omega\}$,  and $\cO_{\om}=\cI_{\om}+\om$ the minimal set of products.  
By  the definition of a minimal set, we have
\begin{equation*}
\underset{\om\in\cT}{\cup}\cI_{\om}\subseteq\cI,\quad \text{cl}(\cI)=\underset{\om\in\cT}{\cup}\text{cl}(\cI_{\om}),
\quad \underset{\om\in\cT}{\cup}\cO_{\om}\subseteq\cO,\quad \text{cl}(\cO)=\underset{\om\in\cT}{\cup}\text{cl}(\cO)_{\om}.
\end{equation*}

Let $\sS\subseteq\R^d$ be the {\em stoichiometric subspace} of an SRN $(\cR,\cK)$, that is, the linear span of the reaction vectors. For $c\in\N^d_0$, define the {\em (stochastic) stoichiometric compatibility class}  (SCC)    of $\cR$ as
\begin{equation*}\label{StocComp}
\sL_c=\lt(\sS+c\rt)\cap\N_0^d.
\end{equation*}
These classes are translational invariant:
$$\sL_c=\sL_{c'},\quad \text{if}\quad c-c'\in\sS\quad \text{and}\quad \sL_c\cap\sL_{c'}=\varnothing,\quad \text{if}\quad c-c'\not\in\sS.$$
Let $\sT_c=\sT\cap\sL_c$, and let $\sP_c$, $\sQ_c$, etc, be defined analogously.

We further characterize the decomposition of the ambient space $\N^d_0$. Define the set of reactants and products for the positive and negative subnetworks, respectively:
 $$\cI_{\pm}=\{y\colon y\to y'\in\cR_{\pm}\},\quad \cO_{\pm}=\{y'\colon y\to y'\in\cR_{\pm}\}.$$
  Furthermore, for $c\in\N_0^d$, let
 $$\sK_c=\sL_c\cap\Bigl(\bigl(\text{cl}(\cI_{+})\cap\text{cl}(\cO_{-})\bigr)\cup\bigl(\text{cl}(\cI_{-})\cap\text{cl}(\cO_{+})\bigr)\Bigr),$$
which is independent of the choice of $c$ due to  translational invariance of $\sL_c$.  Any state in $\sK_c$ can reach another state by a forward jump \emph{and} be reached from some other state in $\sK_c$ by a backward jump, and \emph{vice versa}.  Recall the definition of lattice interval $[x,y]_1$ as well as $\min_1$ and $\max_1$ in Section \ref{sec:lattice}. Let $$\uc={\min}_1\sK_c,\quad \text{and}\quad \bc={\max}_1\sK_c,$$ so that $\sK_c\subseteq [\uc,\bc]_1$. Furthermore,  for $c\in\N_0^d$, the following parameters are used to characterize sets of states of different types on $\sL_c$:
\begin{equation}\label{c-parameters}
\begin{split}
\si(c)=&{\min}_1\sL_c\cap\text{cl}(\cI),\quad \si_+(c)={\min}_1\sL_c\cap\text{cl}(\cI_+),\\
\so(c)=&{\min}_1\sL_c\cap\text{cl}(\cO),\quad \so_-(c)={\min}_1\sL_c\cap\text{cl}(\cO_-).
\end{split}\end{equation}By definition, $\uc\ge_1\max\{\si(c),\so(c)\}$.

The next result provides a detailed characterization of the relevant sets assuming ($\rm\mathbf{H5}$). In the statement below, $\tfrac x y$, for $x,y\in\R^d$, refers to the proportionality constant of $x$ and $y$, which exists due to the assumptions made, see Section \ref{sec:gcd}.

\begin{theorem}\label{th-1d-classification}
Assume \emph{($\mathbf{H1}$)}-\emph{($\mathbf{H4}$)}. Then for every $c\in\N_0^d$,
 $$\uc={\max}_1\lt\{\si_+(c),\so_-(c)\rt\},\quad  \bc=(\bc_1,\ldots,\bc_d)\coloneqq+\infty,\quad \text{with}\ \bc_j=+\infty,\quad j=1,\ldots,d,$$
   and $\sK_c=[\uc, \bc]_1=\sP_c\cup\sQ_c$ consists of all ICs  on $\sL_c$, while
$\sL_c\setminus\sK_c$ is the union of singleton communicating classes, composed of
$$\sN_c=\Bigl\{\min\nolimits_1\sL_c,\ldots,\min\nolimits_1\{\si(c),\so(c)\}-\frac{1}{\om^{**}}\om^*\Bigr\},\quad \sT_c=\Bigl\{\so(c),\ldots,\si(c)-\frac{1}{\om^{**}}\om^*\Bigr\},$$$$\sE_c=\Bigl\{\si(c),\ldots,\max\nolimits_1\{\si_+(c),\so_-(c)\}-\frac{1}{\om^{**}}\om^*\Bigr\}.$$
Furthermore, the  following hold:
\begin{enumerate}
\item If $\#\sT_c=0$, then $\sQ_c=\varnothing$, and  $$\sP_{c}^{s}=\om^*\N_0+\so_-(c)+s\frac{1}{\om^{**}}\om^*,\quad  s = 0,\ldots,\om^{**}-1,$$ are the PICs;
\item If  $\#\sT_c\ge \omega^{**}$,  then  $\sP_c=\varnothing$, and  $$\sQ_{c}^{s}=\om^*\N_0+\si_+(c)+s\frac{1}{\om^{**}}\om^*,\quad s =0,\ldots,\om^{**}-1,$$  are the QICs;
\item If $0<\#\sT_c<\omega^{**}$, then
$$\sQ_{c}^{s}=\om^*\N+\so_-(c)+s\frac{1}{\om^{**}}\om^*,\quad s=0,\ldots,\frac{\si_+(c)-\so_-(c)}{\om^*}-1,$$ are the QICs, and
$$\sP_{c}^{s}=\om^*\N_0+\so_-(c)+s\frac{1}{\om^{**}}\om^*,\quad s=\frac{\si_+(c)-\so_-(c)}{\om^*},\ldots,\om^{**}-1,$$ are the PICs.
\end{enumerate}
In any case, there are $\om^{**}$ PICs and QICs in total.
\end{theorem}

\begin{proof}
We apply \cite[Theorem~3.13,Corollary~3.15]{XHW20a}. The expression of $\sK_c$ follows from \cite[Theorem~3.13]{XHW20a}. The expressions of $c_*$ and $c^*$, as well as those of $\sN_c$, $\sT_c$, and $\sE_c$  follow from \cite[Corollary~3.15]{XHW20a}. It suffices to verify the expressions of ICs in the three different cases. As in \cite{XHW20a}, define the following sets,
   $$\Sigma^+_c=\left\{1+\frac{v-c}{\om^*}\om^{**}-\Bigl\lf\frac{v-c}{\om^*}\Bigr\rf\om^{**}\colon v\in\sT_c\right\},$$
   $$\Sigma^-_c=\left\{1+\frac{v-c}{\om^*}\om^{**}-\Bigl\lf\frac{v-c}{\om^*}\Bigr\rf\om^{**}\colon v\in\{\si(c),\ldots,\so(c)+\om^*-\frac{1}{\om^{**}}\om^*\}\right\}.$$
    Since for $\Bigl\lf\frac{v-c}{\om^*}\Bigr\rf\in\N_0$,
  \[1\le 1+\frac{v-c}{\om^*}\om^{**}-\Bigl\lf\frac{v-c}{\om^*}\Bigr\rf\om^{**}\equiv 1+\frac{v-c}{\om^*}\om^{**}\,\,\text{mod}\,\,\omega^{**} <1+\omega^{**},\]
 and  we have   
$\Sigma_c^+\cap\Sigma_c^-=\emptyset,$ $\Sigma^+_0\cup\Sigma^-_c=\{1,\ldots,\omega^{**}\}$, 
and $\#\Sigma^+_c=\min\{\#\sT_c,\omega^{**}\}$. If $\so(c)<_1 \si(c)$, these conclusions follow easily; if $\so(c)\ge_1\si(c)$, then $\Sigma_c^+=\emptyset$, and the conclusions follow.

 According to \cite[Corollary~3.15]{XHW20a}, it follows that
  \begin{equation}\label{eq:PQ}
\omega^*\left(\N_0+\left\lceil\frac{c_*-c-\frac{k-1}{\om^{**}}\om^{*}}{\omega^*}\right\rceil\right)+\frac{k-1}{\om^{**}}\om^{*}+c=\begin{cases}
    \sP^{(k)}_c,\quad k\in\Sigma^-_c,\\
    \sQ^{(k)}_c,\quad k\in\Sigma^+_c,
  \end{cases}
\end{equation}
are      the disjoint PICs and    the disjoint QICs, respectively, of $(\Omega,\cF)$, in the terminology of \cite{XHW20a}. Consequently,   
 \[\bigcup_{k\in\Sigma^-_c\cup\Sigma^+_c}\left( \sP^{(k)}_c\cup \sQ^{(k)}_c\right)=\sL_c\setminus(\sN_c\cup\sT_c\cup\sE_c)=[c_*,c^*[_1=\sK_c=\frac{\omega^*}{\omega^{**}}\N_0+c_*.\]
Since, for $k\in\Z$,
  \[ c_*\le_1 \omega^*\left\lceil\frac{c_*-c-\frac{k-1}{\om^{**}}\om^{*}}{\omega^*}\right\rceil+\frac{k-1}{\om^{**}}\om^{*}+c<_1 \omega^*+c_*,\]
then   we might state \eqref{eq:PQ} as
  \begin{align*}
 \sP^{(k)}_c&=\left(\omega^*\Z+\frac{k-1}{\om^{**}}\om^{*}+c\right)\cap\left(\frac{1}{\omega^{**}}\omega^*\N_0+c_*\right),\quad \text{for} \quad k\in\Sigma^-_c, \\
   \sQ^{(k)}_c&=\left(\omega^*\Z+\frac{k-1}{\om^{**}}\om^{*}+c\right)\cap\left(\frac{1}{\omega^{**}}\omega^*\N_0+c_*\right),\quad \text{for} \quad k\in\Sigma^+_c.
  \end{align*}

We  show that the expressions given for $\sP^{(k)}_c,\sQ^{(k)}_c$ correspond to those given for $\sP_{c}^{s},\sQ_{c}^{s}$ in the three cases, by suitable renaming of the ICs.
First, note that $\sT_c=\varnothing$ if and only if $\so(c)\ge_1\si(c)$.
 From $\so(c)\le_1\so_-(c)<_1\si_-(c)\coloneqq{\min}_1\sL_c\cap\text{cl}(\cI_-)$ and $\so(c)\ge_1\si(c)$,   we have $\si(c)=\si_+(c)\le_1\so_-(c)$, which yields   $c_*=\so_-(c)$. Consequently, $\Sigma_c^+=\emptyset$ and $\sQ_c=\varnothing$. This proves the expression for  $\sP_{c}^{s}$ in (1).
 
  Otherwise, if $\so(c)<_1\si(c)$, then $\so(c)<_1\si(c)\le_1\si_+(c)<_1\so_+(c)\coloneqq{\min}_1\sL_c\cap\text{cl}(\cO_+)$, which implies that $\so(c)=\so_-(c)<_1\si_+(c)$. Hence, $c_*=\si_+(c)$. If $\#\sT_c\ge\om^{**}$, then $\sP_c=\varnothing$, which proves  the expression for $\sQ_{c}^{s}$ in (2).
It remains to prove the last case. If $0<\#\sT_c<\omega^{**}$, then $\#\sT_c=\#\Sigma^+_c$, and for every $v\in\sT_c$, 
\begin{align*}
\sQ_c^{(k)}=\left(\omega^*\Z+v\right)\cap\left(\frac{1}{\omega^{**}}\omega^*\N_0+c_*\right),
 \end{align*}
with $k=1+\frac{v-c}{\om^*}\om^{**}-\Bigl\lf\frac{v-c}{\om^*}\Bigr\rf\om^{**}\in\Sigma^+_c$.
 If $\si(c)=\si_+(c)$, then using the above equation and  $\so(c)=\so_-(c)$, the expression for  $\sQ_{c}^{s}$ in (3) follows directly,  and the remaining ICs must be PICs. Finally, we show $\si(c)<_1\si_+(c)$ is impossible, which concludes the proof. Assume oppositely that $\si(c)<_1\si_+(c)$. Then, $\si(c)=\si_-(c)$, $\sT_c=\{\so_-(c),\ldots,\si_-(c)-\frac{1}{\omega^{**}}\omega^*\}$, and $\si_-(c)\in \sE_c$. This implies one can jump from the state $\si_-$ (smallest state for which a backward jump can be made) leftwards to a state $x\le \si_-(c)-\omega^*< \so_-(c)$. The latter inequality comes from $0<\#\sT=\frac{\si(c)-\so(c)}{\omega^*}<\omega^{**}$ and $\so(c)=\so_-(c)$. However, this implies $x\in\sN_c$, which is impossible.

The total number of PICs and QICs follows from $\Sigma^+_c\cup\Sigma^-_c=\{1,\ldots,\omega^{**}\}$. The indexation follows  from $c_*=\max_1\{\si_+(c),\so_-(c)\}$ in the two case (1) and (3). Also, the inequality $\si_+(c)<_1\so_-(c)+\om^*$ follows straightforwardly in these two cases. It remains to check it is not fulfilled in case (2). In that case, $\#\sT_c=\frac{\si(c)-\so(c)}{\omega^*}\ge\omega^{**}$ by assumption, hence $\si_+(c)\ge_1\si(c)\ge_1\so(c)+\om^*=\so_-(c)+\om^*$, and the conclusion follows.
\end{proof}

\section{Threshold dynamics}

 For ease of statements, we define dynamical properties of an SRN in terms of its underlying CTMCs.

\begin{definition}\label{def-dyn}
For  $c\in\N^d_0$, we say an SRN $(\cR,\cK)$  
\begin{enumerate}
\item[(i)]  is \emph{non-explosive}  \emph{(explosive)}  on $\sL_c$ if every CTMC  with initial state  in $\sP_c\cup\sQ_c$ is so,
\item[(ii)]  is \emph{(positive/null) recurrent} \emph{(transient)}  on $\sL_c$ if $\sP_c\neq\varnothing$ and all CTMCs on PICs of $\sP_c$ are so, %
\item[(iii)]  is \emph{(exponentially) ergodic} on $\sL_c$ if $\sP_c\neq\varnothing$ and all   CTMCs on PICs of $\sP_c$ have  (exponential) ergodic stationary distributions,
\item[(iv)] admits \emph{extinction events a.s.}\,on $\sL_c$ if $\sQ_c\neq\varnothing$  and every CTMC  with initial state  in $\sQ_c$ has certain absorption into $\sT_c\cup\sE_c$,
\item[(v)] is \emph{(uniformly) (exponentially) quasi-ergodic} on $\sL_c$ if $\sQ_c\neq\varnothing$ and all CTMCs have (uniformly) (exponential) ergodic QSDs on  QICs of $\sQ_c$,
\item[(vi)] is \emph{not quasi-ergodic} on $\sL_c$ if $\sQ_c\neq\varnothing$ and no   CTMCs have   ergodic QSDs on any QIC of $\sQ_c$.
\end{enumerate}
\end{definition}

The definitions exclude the possibility that an SRN is explosive on one communicating class within some SCC,  while   non-explosive on another communicating class in the same SCC, or that it  is ergodic on one PIC while not on another PIC is the same SCC,
  and so on. This is due to the fact  that the dynamics is characterized by   parameters that are insensitive to the concrete class but not its type (as long as it is in the same SCC), see below.

Let $R=\max_{y\to y'\in\cR} \|y\|_1$ be the \emph{order} of an SRN   $(\cR,\cK)$, where we for convenience assume $\cK$ is stochastic mass-action kinetics \cite{AK15}. 
 The following parameters are well defined in terms of \emph{directional limits} under $(\rm\mathbf{H2})$,  for  any  $c\in\N_0^d$:
\begin{equation*}\label{Alpha}
\alpha_c=\lim_{\substack{x_1\to\infty \\ x\in\sL_c}}\frac{\sum_{y\to y'\in\cR}\kappa_{y\to y'}x^{\underline{y}}(y'_1-y_1)}{x_1^R},
\end{equation*}
\begin{equation*}\label{Beta'}
\beta_c=\lim_{\substack{x_1\to\infty \\ x\in\sL_c}}\frac{\sum_{y\to y'\in\cR}\kappa_{y\to y'}x^{\underline{y}}(y'_1-y_1)-\alpha_c x_1^R}{x_1^{R-1}}-\frac{1}{2}\lim_{\substack{x_1\to\infty \\ x\in\sL_c}}\frac{\sum_{y\to y'\in\cR}\kappa_{y\to y'}x^{\underline{y}}(y'_1-y_1)^2}{x_1^{R}}.
\end{equation*} 

 One can understand $\alpha_c$ as a stability index of the underlying Markov chains associated with the SRN, which determines the stochastic stability of the Markov chains \cite{MT09} in the non-critical case ($\alpha_c\not=0$). An arguably good analogue of $\alpha_c$ is the Lyapunov exponent of a linear stochastic differential equation, which determines the almost sure stability of solutions to zero \cite[Section~6.7]{K12}. Furthermore, $\beta_c$ can be understood as a stability index in the critical case where $\alpha_c=0$. Indeed, both parameters above can be represented in terms of the reactions, the reaction rate constants as well as  the stochastic stoichiometric compatibility class. 

\begin{example}
  Consider the following one-species  SRN with mass-action kinetics, which is not of BDP type: 
  \begin{equation}\label{ex-alpha-beta}
    \tS\ce{<=>[\kappa_1][\kappa_2]}2\tS\ce{->[\kappa_3]}4\tS
  \end{equation}
where the labels are the reaction rate constants.
  For this SRN,  $\sL_c=N_0$, $c\in\N_0$, is the unique SCC. Moreover, $\cT=\{-1,1,2\}$ with   rate functions,
  \[\lambda_{-1}(x)=\kappa_2x(x-1),\quad \lambda_1(x)=\kappa_1x,\quad \lambda_2(x)=\kappa_3x(x-1),\quad x\in\N_0.\]
  Hence, $R=2$, and
  \begin{align*}
  \alpha_c&= \lim_{x\to\infty}\frac{-\kappa_2x(x-1) +2\kappa_1x+\kappa_3x(x-1)}{x^2} 
  =-\kappa_2+2\kappa_3,  \\
    \beta_c &=\lim_{x\to\infty}\frac{-\kappa_2x(x-1)+\kappa_1x+2\kappa_3x(x-1) -(2\kappa_3-\kappa_2)x^2}{x}\\
&\quad -\frac{1}{2}\lim_{x\to\infty}
    \frac{\kappa_2x(x-1)+\kappa_1x+4\kappa_3x(x-1)}{x^2}\\
     &=\kappa_1-(-\kappa_2+2\kappa_3)-\frac{1}{2}(\kappa_2+4\kappa_3)=\kappa_1-\frac 12 \kappa_2-2\kappa_3. 
  \end{align*}
 The index   $\alpha_c$ only depends on reactions of order $R$, while $\beta_c$   depends on reactions of order $R$ and $R-1$.  
 \end{example}
  
 \begin{proposition}\label{prop-formula} 
 Let $(\cR,\cK)$ be an SRN with stochastic mass-action kinetics, satisfying  $\rm{(\mathbf{H2})}$-$\rm{(\mathbf{H5})}$.
  For  $c\in\N_0^d$,
\[\alpha_c=(\om^*_1)^{-R}\sum_{\|y\|_1=R}(y'_1-y_1)\kappa_{y\to y'}\prod_{l=1}^d\lt(\om^*_j\rt)^{y_j},\]
\begin{multline*}\beta_c=(\om^*_1)^{-R+1}\sum_{\|y\|_1=R-1}(y'_1-y_1)\kappa_{y\to y'}\prod_{j=1}^d\lt(\om^*_j\rt)^{y_j}\\-\frac{1}{2}(\om^*_1)^{-R}\sum_{\|y\|_1=R}(y'_1-y_1)^2\kappa_{y\to y'}\prod_{j=1}^d\lt(\om^*_j\rt)^{y_j}\\
+(\om^*_1)^{-R}\sum_{\|y\|_1=R}(y'_1-y_1)\kappa_{y\to y'}
\prod_{j=1}^d\lt(\om^*_j\rt)^{y_j}\left(\sum_{j=1}^dy_j\lt(\frac{\om^*_1}
{\om^*_j}\lt(c_j-\frac{y_j-1}{2}\rt)-c_1\rt)\right),\end{multline*}
In particular, $\alpha_c$ is independent of $c$.
 \end{proposition}

\begin{proof}
For any $x\in\sL_c$, 
\begin{equation}\label{Relation}
\frac{x_1-c_1}{\om^*_1}=\frac{x_j-c_j}{\om^*_j},\quad \text{for}\quad j=1,\ldots,d.\end{equation}
Moreover, observe that for any $y\to y'\in\cR$, $x\to\infty,\ x\in\sL_c$, due to ($\rm\mathbf{H3}$), \begin{equation}\label{limit_relation}\frac{x_j^{\underline{y_j}}}{x_1^{y_j}}\to\left(\frac{\om_j^*}{\om_1^*}\right)^{y_j},\quad \text{for}\quad j=1,\ldots,d.\end{equation}
 In the light of the limit definitions of these parameters, and the functions inside the limits are rational, substituting \eqref{Relation} and \eqref{limit_relation} into the limit definitions of $\alpha_c$ and $\beta_c$ we can obtain these formulas in a tedious but straightforward way by comparing the coefficients of the leading term in the asymptotic expansions of these rational functions for large $x$.
\end{proof}

The parameters do not depend on the specific choice of $c$  as long as $\sL_c=\sL_{c'}$. Furthermore, since $\alpha_c$ is  independent of $c$ from Proposition~\ref{prop-formula}, we  omit $c$ in the notation hereafter. The results concerning the dynamics in Theorem \ref{th-Threshold} only depends on $\alpha,\beta_c$ and $R$. {This is in contrast to the more general results in \cite{XHW20b} for a class of Markov Chains with power-law expansions of the transition rates, where the dynamical classification depends on an extra parameter in the critical regime. This parameter is obsolete in the case of mass-action kinetics, where the transition rates  in particular are polynomial, simplifying the classification. }

In the non-critical case when $\alpha\not=0$, the classification of the dynamics is the same irrespective of 
   the SCC, and only depends on the sign of $\alpha$ and $R$. In the critical case ($\alpha=0$), the classification additionally depends on $\beta_c$ and it is possible to have a phase transition, where $\beta_c$ changes from negative, to zero and   positive values as $c$ (representing the SCC) changes, see Example \ref{ex-bifurcation}. Thus, in such cases, the dynamics is sensitive to the initial condition of the Markov chain. We do not address here explicitely  the possibilty of phase transitions as the rate constants (the $\kappa_i$s) are changed. This will be addressed in a forth coming paper.

\begin{theorem}\label{th-Threshold}
 Let $(\cR,\cK)$ be an SRN with stochastic mass-action kinetics, satisfying 
$\rm{(\mathbf{H2})}$-$\rm{(\mathbf{H5})}$. Let $c\in\N^d_0$.
\begin{enumerate}
  \item[(i)] $(\cR,\cK)$ is explosive on $\sL_c$ if and only if either {\rm(i-a)} $R>1$, $\alpha>0$, or {\rm(i-b)} $R>2$, $\alpha=0$, $\beta_c>0$.
  \item[(ii)] Assume $\sP_c\neq\varnothing$, then $(\cR,\cK)$ is
    \begin{enumerate}
  \item[(ii-a)] recurrent on $\sL_c$ if either {\rm(a-1)} $\alpha<0$, or {\rm(a-2)} $\alpha=0$, $\beta_c\le0$, while it is transient otherwise.
  \item[(ii-b)] positive recurrent  on $\sL_c$ if and only if either {\rm(a-1)}, {\rm(b-1)} $\alpha=0$, $\beta_c=0$, $R>2$, or {\rm(b-2)} $\alpha=0$,  $\beta_c<0$, $R>1$,
  \item[(ii-c)]  null recurrent  on $\sL_c$ if and only if either {\rm(c-1)} $\alpha=0$, $\beta_c\le0$, $R=1$,  or {\rm(c-2)} $\alpha=0$, $\beta_c=0$, $R=2$, 
  \item[(ii-d)] exponentially ergodic on $\sL_c$  if {\rm(a-1)} or {\rm(b-1)} holds.
\end{enumerate}
\item[(iii)] Assume $\sQ_c\neq\varnothing$, then $(\cR,\cK)$
\begin{enumerate}
  \item[(iii-a)] admits extinction events a.s.\,on $\sL_c$ if and only if either {\rm(a-1)} or {\rm(a-2)} holds.
  \item[(iii-b)] is uniformly exponentially quasi-ergodic  on $\sQ_c$ if either {\rm(b-1)} or {\rm(a-1)'} $\alpha<0$, $R>1$, and is not quasi-ergodic on $\sL_c$ if none of {\rm(a-1)}, {\rm(b-1)}, {\rm(b-2)} holds.
\end{enumerate}
\end{enumerate}
\end{theorem}
\captionsetup[table]{singlelinecheck=off}
\begin{table}[tbhp]
\begin{center}
\setlength{\tabcolsep}{5pt}
\renewcommand{\arraystretch}{1.3}
\small{\begin{tabular}{  |p{1.05cm}|p{.75cm}|p{1cm}|p{1cm}|p{1cm}|p{.75cm}| }
\hhline{~|-|-|-|-|-|}
\multicolumn{1}{c|}{}&\multirow{2}{50em}{$\alpha<0$}&\multicolumn{3}{c|}{$\alpha=0$}
&\multirow{2}{50em}{$\alpha>0$}
\\\hhline{~|~|-|-|-|~|}
\multicolumn{1}{c|}{}&&$\beta_c<0$&$\beta_c=0$&$\beta_c>0$& \multicolumn{1}
{c|}{}\\\hline
$R=0$ &\multicolumn{4}{c|}{\cellcolor{black}}&\multicolumn{1}{c|}
{\cellcolor{green}} \\\hhline{|-|-|-|-|-|~|}
$R=1$ &\multicolumn{1}{c|}{\cellcolor{red}}&\multicolumn{2}{c|}{\cellcolor{blue}NS}&\multicolumn{2}{c|}{\cellcolor{green}NS} \\\hhline{|-|~|~|~|~|-|}
$R=2$ &\multicolumn{2}{c|}{\cellcolor{red}ES}&\multicolumn{1}{c|}{\cellcolor{blue}}&
\multicolumn{1}{c|}{\cellcolor{green}} &\multicolumn{1}{c|}{\cellcolor{yellow}} \\\hhline{|-|~|~|-|-|~|}
$R>2$ &\multicolumn{3}{c|}{\cellcolor{red}}&\multicolumn{2}{c|}{\cellcolor{yellow}} \\\hline
\end{tabular}}
\caption[]{\small Summary of parameter regions for dynamics. Respective properties hold in connected regions (with appropriate provisions for the initial state on $\sL_c$).   Red: Positive recurrent. Blue: Null recurrent. Gray: Recurrence of unknown type. Green: Transient and non-explosive.  Yellow: Explosive. Black: Empty set. ES=exponential ergodicity of stationary distribution. NS=no stationary distributions.}
\end{center}
\end{table}

\begin{proof}
Let $Y_t$ with $Y_0\in\N^d_0$ be a CTMC associated with $\cR$.
Given any $c\in\N^d_0$ and let $Y_0\in\sL_c\setminus(\sN\cup\sT)$. Then $Y_t$ is a one-dimensional CTMC on $\sL_c$.
Denote the $Q$-matrix associated with $Y_t$ on $\sL_c$ by $Q=(q_{zz'})_{z, z'\in\sL_c}$, where
$$q_{xx'}=\begin{cases}
  \sum_{y'-y=x'-x}\kappa_{y\to y'}x^{\underline{y}},\ \text{if}\ x'-x\in\cT,\\
  0,\hspace{3.15cm} \text{else},
\end{cases}\quad x\neq x', x,\ x'\in\sL_c.$$
 Let $X_t=Y_{t,1}$ be the first coordinate of $Y_t$, for $t\ge0$. Then $Y_t=\frac{\om^*}{\om^*_1}X_t$, and the dynamics of $Y_t$ and $X_t$ are consistent in the sense that $X_t$ has some dynamical property (e.g., recurrence) on $(\sL_{c})_1\subseteq\N_0$ (the projection of $\sL_c$ to the first axis) if and only if $Y_t$ does so on $\sL_c$.
   Let $\widetilde{Q}=(\widetilde{q}_{ww'})_{w, w'\in (\sL_{c})_1}$ be the $Q$-matrix associated with $X_t$ on $(\sL_{c})_1$:
 $$\widetilde{q}_{z_1z'_1}=q_{zz'},\quad  z, z'\in\sL_c.$$
Hence the definitions of the parameters $R$, $\alpha$ and $\beta_c$ are consistent with those for $X_t$ on $\N_0$ given in \cite{XHW20b}. In the following, we apply results in \cite{XHW20b} to $X_t$.

  (i) By Proposition~\ref{prop-Imp}, $\alpha<0$ implies $R\ge1$. Then the conclusions follow from \cite[Theorem~3.1]{XHW20b}. {By Definition~\ref{def-dyn}, $Y_0$ is in either a PIC or a QIC. If $Y_0$ is in a PIC, then the explosivity of $Y_t$ follows from \cite[Theorem~3.1]{XHW20b} for CTMCs on an irreducible state space; if  $Y_0$ is in a QIC, then the explosivity of $Y_t$ again follows from \cite[Theorem~3.1]{XHW20b}, but for CTMCs on a state space  having an absorbing set $A=\sN_c\cup\sE_c$ as well as the QIC.}
 
  (ii) Applying \cite[Theorem~3.3(i)]{XHW20b} to $X_t$ yields (ii-a). By Proposition~\ref{prop-Imp} and $\sP_c\neq\varnothing$, the following two cases never {appear: (1) $\alpha\le0$, $R=0$, and (2) $\alpha=0$, $\gamma_c\le0$, $R=1$, where \begin{equation}\label{Gamma}
\gamma_c=\lim_{\substack{x_1\to\infty \\ x\in\sL_c}}\frac{\sum_{y\to y'\in\cR}\kappa_{y\to y'}x^{\underline{y}}(y'_1-y_1)-\alpha x_1^R}{x_1^{R-1}}.
\end{equation}}
Hence $R>0$ whenever $\alpha\le0$, and when $\cR$ is recurrent, by (ii-a), none of (a-1), (b-1), (b-2) holds if and only if either (c-1) or (c-2) holds. 
Applying \cite[Theorem~3.8(i)]{XHW20b} to $X_t$ yields (ii-b)-(ii-d), respectively.

 (iii) Let $Y_0\in\sQ_c$ and hence $X_0\in(\sQ_c)_1$. Note that quasi-ergodicity implies existence of QSDs. Then the conclusion follows by applying \cite[Theorem~3.8(ii)]{XHW20b} to $X_t$.
\end{proof}

For more applications of Theorem~\ref{th-Threshold} to biological examples, see \cite[Section~4]{XHW20b}. Applications to questions of interest in stochastic reaction network theory are demonstrated below and in the next subsection.

Structurally identical SRNs can have opposite dynamics as illustrated below.

\begin{example}\label{ex-explosive}
Consider the following two one-species SRNs with   mass-action kinetics, both of BD type, discussed in \cite{ACKK18}:
\begin{equation*}
\tS\ce{<=>[1][2]} 2\tS\ce{<=>[7][4]} 3\tS\ce{<=>[6][1]} 4\tS\ce{->[1]} 5\tS,\qquad \tS\ce{<=>[1][2]}  2\tS\ce{<=>[3][1]}  3\tS\ce{->[1]}  4\tS.
\end{equation*}
 It is easy to verify that their common ambient space $\N_0$ is decomposed into $\sP=\N$ and $\sN=\{0\}$.

It is easy to see they are BDPs with birth and death rates having the same leading term for large states, which implies that the embedded chain jumping to the left and to the right with $1/2-o(1)$ probability for large states. Moreover, the underlying CTMCs of both SRNs are asymptotic symmetric random walks on $\N_0$ with a left reflecting boundary. It is straightforward to calculate that $\alpha=0$, $\beta=1$, $R=4$ for the first SRN, while $\alpha=0$, $\beta=0$, $R=3$ for the second. By Theorem~\ref{th-Threshold}, the first is explosive while the second is positive recurrent in $\N_0$.
\end{example}

With Theorem~\ref{th-Threshold} in mind, one can construct deterministically identical but stochastically different reaction networks.

\begin{example}\label{ex-4}
(i) Consider the following reaction network:
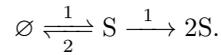
\[\varnothing\ce{<=>[1][2]}\tS\ce{->[1]}2\tS.\]
The deterministic system (modelled as an ODE with mass-action kinetics and reaction rate constants as given in the reaction graph) has a unique globally asymptotically stable positive steady state $x_*=1$, and the SRN is positive recurrent by Theorem~\ref{th-Threshold}. Now, add a pair of reactions:
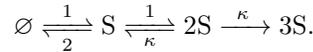
\[\varnothing\ce{<=>[1][2]}\tS\ce{<=>[1][\kappa]}2\tS\ce{->[\kappa]}3\tS.\]
This modified reaction network preserves the deterministic dynamics as well as one-step reachability among states of the underlying CMTCs. Nevertheless, by Theorem~\ref{th-Threshold}, this SRN is explosive if $\kappa<1$ and positive recurrent if $\kappa\ge1$. Hence, with $\kappa<1$, we \emph{destablize} the original SRN (in the sense of ergodicity).

(ii) Consider a similar reaction network as in (i):
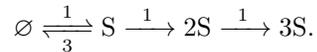
\[\varnothing\ce{<=>[1][3]}\tS\ce{->[1]}2\tS\ce{->[1]}3\tS.\]
This deterministic reaction network has a unique unstable positive steady state $x_*=1$, and the SRN is explosive by Theorem~\ref{th-Threshold}. Now, add a pair of reactions:
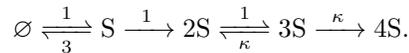
\[\varnothing\ce{<=>[1][3]}\tS\ce{->[1]}2\tS\ce{<=>[1][\kappa]}3\tS\ce{->[\kappa]}4\tS.\]
This new reaction network preserves the deterministic dynamics  as well as one-step reachability among states  of the underlying CMTCs. However, by Theorem~\ref{th-Threshold} this SRN is explosive if $\kappa<1$, and positive recurrent if $\kappa\ge1$. Hence with $\kappa\ge1$, we \emph{stablize} the original SRN.
\end{example}
The   dynamics may change with the stochastic stoichiometric compatibility classes.
\begin{example}\label{ex-bifurcation}
   Consider the following mass-action SRN $\cR$:
\[\varnothing\ce{<=>[\kappa_1][\kappa_{2}]}2\tS_1+2\tS_2,\quad 4\tS_2\ce{->[\kappa_2]}2\tS_1+6\tS_2.\]
 It is readily proved that $\cR$ is essential. Moreover, for $c\in\N^2_0$, $\sL_c=\sP_c$ consists of two PICs. Furthermore, $R_+=R_-=4$, and by Proposition~\ref{prop-formula},  it is straightforward to calculate that $\alpha=0$ and $\beta_c=4\kappa_2(c_2-c_1-3)$, for $c\in\N^2_0$.
It follows from Theorem~\ref{th-Threshold} that $\cR$ is explosive on $\sL_c$ for all $c\in\N^2_0$ with $c_2-c_1>3$, while is positive recurrent on $\sL_c$ for all $c\in\N^2_0$ with $c_2-c_1\le3$.  Hence, a phase transition occurs when stoichiometric compatibility classes passes the critical class $\sL_c$ for $c_2-c_1=3$. See Figure~\ref{fig_ex-bifurcation}.

\begin{figure}[h]
\begin{center}
\scalebox{0.7}{\begin{tikzpicture}
    [
        dot/.style={circle,draw=black,fill,inner sep=1pt},scale=0.6
    ]
\foreach \x in {0,...,8}
 \foreach \y in {0,...,\x}
\node[dot,thick,orange] at (\x,\y){} edge [thin,-latex,bend left=-10,orange] (\x+2,\y+2){};

\foreach \x in {2,...,10}
 \foreach \y in {2,...,\x}
\node[dot,thick,orange] at (\x,\y){} edge [thin,-latex,bend left=-10,orange] (\x-2,\y-2){};

\foreach \x in {1,...,8}
 \foreach \y in {1,...,\x}
\node[dot,thick,orange] at (\x-1,\y){} edge [thin,-latex,bend left=-10,orange] (\x+1,\y+2){};

\foreach \x in {3,...,10}
 \foreach \y in {3,...,\x}
\node[dot,thick,orange] at (\x-1,\y){} edge [thin,-latex,bend left=-10,orange] (\x-3,\y-2){};

\foreach \x in {2,...,8}
 \foreach \y in {2,...,\x}
\node[dot,thick,orange] at (\x-2,\y){} edge [thin,-latex,bend left=-10,orange] (\x,\y+2){};

\foreach \x in {4,...,10}
 \foreach \y in {4,...,\x}
\node[dot,thick,orange] at (\x-2,\y){} edge [thin,-latex,bend left=-10,orange] (\x-4,\y-2){};

\foreach \x in {3,...,8}
 \foreach \y in {3,...,\x}
\node[dot,thick,orange] at (\x-3,\y){} edge [thin,-latex,bend left=-10,orange] (\x-1,\y+2){};

\foreach \x in {5,...,10}
 \foreach \y in {5,...,\x}
\node[dot,thick,orange] at (\x-3,\y){} edge [thin,-latex,bend left=-10,orange] (\x-5,\y-2){};


\foreach \y in {4,...,8}
 \foreach \x in {4,...,\y}
\node[dot,thick,green] at (\x-4,\y){} edge [thin,-latex,bend left=-10,green] (\x-2,\y+2){};

\foreach \y in {6,...,10}
 \foreach \x in {6,...,\y}
\node[dot,thick,green] at (\x-4,\y){} edge [thin,-latex,bend left=-10,green] (\x-6,\y-2){};

\foreach \x in {1,...,9}
    \draw (\x,.1) -- node[below,yshift=-1mm] {\x} (\x,0);
\node[below,xshift=-2mm,yshift=-1mm] at (0,0) {0};

\foreach \y in {1,...,9}
    \draw (.1,\y) -- node[below,xshift=-3mm,yshift=2mm] {\y} (0,\y);
\draw[->,thick,-latex] (0,-1) -- (0,10);
\draw[->,thick,-latex] (-1,0) -- (10,0);
\draw[dashed,thick,blue] (0,3) -- (7,10);
\draw[step=1cm,gray,very thin,dashed] (0,0) grid (10,10);

\node[below,xshift=-2mm,yshift=-1mm] at (10,0) {$c_1$};
\node[below,xshift=-2mm,yshift=-1mm] at (0,10) {$c_2$};
  \end{tikzpicture}
}
\caption{\small Illustration of Example~\ref{ex-bifurcation}. Blue: Bifurcation line. Green: Positive recurrent classes. Orange: Explosive  classes.}\label{fig_ex-bifurcation}
\end{center}
\end{figure}
\end{example}

\subsection*{Ergodicity of (strongly) endotactic and weakly reversible SRNs}

Given a mass-action SRN $\cR$, let $R_+=\max_{y\to y'\in\cR_+}\|y\|_1$ and $R_-=\max_{y\to y'\in\cR_-}\|y\|_1$ be the orders of $\cR_+$ and $\cR_-$, respectively.

\begin{theorem}\label{th-endotactic}
Let $\cR$ be an  endotactic SRN with mass-action kinetics, and satisfying $\rm{(\mathbf{H2})}$-$\rm{(\mathbf{H5})}$. Then $R_->R_+$. Furthermore, for  $c\in\N^d_0$, $\cR$ is non-explosive on $\sL_c$, and
\begin{enumerate}
\item[(i)] is exponentially ergodic on $\sL_c$, if $\sP_c\neq\varnothing$,
\item[(ii)] is uniformly exponentially ergodic on $\sL_c$, if $R>1$ and $\sQ_c\neq\varnothing$.
\end{enumerate}
\end{theorem}

\begin{proof}
By Theorems~\ref{th-Threshold}, it suffices to show $R_->R_+$, which implies $R\ge1$ and $\alpha(c)<0$ for all $c\in\N^d_0$.
Let $\mathbf{1}\in\N^d$ be the vector with all coordinates being 1. Hence, for $y\in\N^d_0$, $\|y\|_1=\la\mathbf{1},y\ra$. Since $\{\frac{y'-y}{\om^*}\colon y\to y'\in\cR_+\}\subseteq\N$ and $\{\frac{y'-y}{\om^*}\colon y\to y'\in\cR_-\}\subseteq(-\N)$, by $\rm{(\mathbf{H5})}$, we have
$$\la \mathbf{1}, y'-y\ra>0,\quad \text{for}\quad y\to y'\in\cR_+,\quad \text{and}\quad \la \mathbf{1}, y'-y\ra<0,\quad \text{for}\quad  y\to y'\in\cR_-,$$ respectively.
Consequently,  $\la \mathbf{1}, y'-y\ra\neq0,$ for all $y\to y'\in\cR,$
 that is, all reactions in $\cR$ are \emph{$\mathbf{1}$-essential} by Definition~\ref{def-endotactic}. Since $\cR$ is \emph{endotactic},  it is $\mathbf{1}$-endotactic. By Definition~\ref{def-max}, every \emph{$\le_{\mathbf{1}}$-maximal element} of the set $\{y\colon y\to y'\in\cR\}$  is a reactant of a reaction in $\cR_-$, which implies
  $$\la\mathbf{1}, \widetilde{y}\ra<\max_{y\to y'\in\cR}\la\mathbf{1}, y\ra,\quad \text{for all}\quad \widetilde{y}\to\widetilde{y}'\in\cR_+,$$
  that is, $\|\widetilde{y}\|<\max_{y\to y'\in\cR}\|y\|,$ for all $\widetilde{y}\to\widetilde{y}'\in\cR_+.$
   Hence $R_+<R$, that is, $R_-=R>R_+$.
\end{proof}

From the proof,  the conclusions hold for $\mathbf{1}$-endotactic SRNs. In particular,  if $\cR$ is $\mathbf{1}$-endotactic, then $R_->R_+$. Indeed, the converse is also true.

\begin{theorem}\label{prop-end}
  Let $\cR$ be a mass-action SRN. Assume $(\rm\mathbf{H2})$-$(\rm\mathbf{H4})$. If $R_->R_+$, then $\cR$ is $\mathbf{1}$-endotactic.
\end{theorem}

\begin{proof}
  From the proof of Theorem~\ref{th-endotactic}, every reactant is $\mathbf{1}$-essential. Since $R=R_->R_+$, \[R_+=\max_{y\to y'\in\cR_+}\la\mathbf{1},y\ra<\max_{y\to y'\in\cR}\la\mathbf{1},y\ra,\] every $\le_{\mathbf{1}}$-maximal reactant is one of a reaction in $\cR_-$. By definition, $\cR$ is  $\mathbf{1}$-endotactic \cite{GMS14}.
\end{proof}

We list a number of remarks:

$\bullet$ In contrast, $R_->R_+$ does not imply $\cR$ is endotactic when $d>1$. Consider, e.g., the 2-species mass-action reaction network:
\[2\eS_1+2\eS_2\ce{->[\kappa_1]}\eS_1,\quad 3\eS_2\ce{->[\kappa_2]}\eS_1+5\eS_2,\] 
where $\kappa_1,\ \kappa_2>0$. For this SRN, $$\cR_-=\{2\eS_1+2\eS_2\ce{->[\kappa_1]}\eS_1\}\quad \text{and}\quad \cR_+=\{3\eS_2\ce{->[\kappa_2]}\eS_1+5\eS_2\},$$ $\cT=\{\om^*,-\om^*\}$ with $\om^*=(1,2)$. Hence, both reactions are $\om^*$-essential, and $R_-=4>R_+=3$. However, $\la\om^*,(2,2)\ra=\la\om^*,(0,3)\ra=6$, which means both reactants are $\le_{\om^*}$-maximal. Hence, $\cR$ is \emph{not} $\om^*$-endotactic, and thus not endotactic.

$\bullet$ The condition that $R>1$ for the ergodicity of QSDs is crucial. Consider the one-species SRN $\cR$:
\[\eS\ce{->[\kappa_1]}\varnothing.\]
It is easy to verify that $\cR$ is endotactic, while there exist a continuum family of QSDs supported on the unique QIC $\N$ trapped to $0$ \cite{V91}.

$\bullet$ This result \emph{cannot} be extended to higher dimensions. See \cite{ACKN18,AC19,AM18} for constructed explosive strongly endotactic SRNs.

$\bullet$ The converse of Theorem~\ref{th-endotactic} is not true. Indeed, non-explosivity does not imply that the SRN is endotactic, consider e.g., Example~\ref{ex-4}(i).

Since  weakly reversible SRNs and strongly endotactic SRNs in particular are  endotactic \cite{GMS14}, the conclusions in Theorem~\ref{th-endotactic} hold for weakly reversible and strongly endotactic  SRNs as well. {Recall that $\rm{(\mathbf{H1})}$ and $\rm{(\mathbf{H3})}$ hold for  weakly reversible mass-action SRNs.}

\begin{corollary}\label{cor-posrec}
Let $\cR$ be a weakly reversible mass-action SRN. Assume $\rm{(\mathbf{H2})}$-$\rm{(\mathbf{H3})}$ are fulfilled. Then $\cR$ is positive recurrent on every PIC with an exponentially ergodic stationary distribution. In particular, the positive recurrence conjecture \cite{AK18} holds in one dimension.
\end{corollary}

There are other partial results in the direction of the positive recurrence conjecture for binary SNRs with additional conditions, in particular for strongly endotactic SRNs  \cite{ACKN18} and weakly reversible SRNs  \cite{ACK19}.  In comparison, our results are not restricted to binary SRNs.

\section{Pattern of stationary distributions}

Stationary distributions are difficult to characterize for SRNs and are only known in few special cases. If an SRN in addition is a BDP then the stationary distribution can be found. If an SRN additionally is \emph{complex balanced} \cite{GST19} then the stationary distribution has a Poisson product-form.
In the following, we use \emph{tail distributions} to describe certain properties of stationary distributions and QSDs of SRNs. {In particular, regarding decay of tails of stationary distributions and QSDs, endotatic reaction networks can be viewed as a generalization of complex balanced reaction networks, since endotatic SRNs have tails decaying as \emph{Conley-Maxwell-Poisson} distributions, which are generalized Poisson distributions.}

The \emph{Conley-Maxwell-Poisson} (CMP) distribution on $\N_0$ with parameter  $(a,b)\in\R_{>0}^2$ has probability mass function given by \cite{JKK05}:
\[{\sf CMP}_{a,b}(x)=\frac{a^x}{(x!)^b}\lt(\sum_{j=0}^{\infty}\frac{a^j}{(j!)^b}\rt)^{\!\!\!-1},\quad x\in\N_0.\]
In particular, ${\sf CMP}_{a,1}$ is a Poisson distribution. The Zeta distribution on $\N_0$ with parameter $a>1$ has probability mass function given by \cite{JKK05}:
\[{\sf Zeta}_a(x)=\zeta(a)^{-1}x^{-a},\]
where $\zeta(a)=\sum_{i=1}^{\infty}i^{-a}$ is the Riemann zeta function of $a$.

Let $\mu$ be a probability distribution and $T_{\mu}\colon\N_0\to[0,1]$, {$T_{\mu}(x)=\sum_{y\ge x}\mu(y)$} its tail distribution.
We say $\mu$ has \emph{a CMP-like tail} if $T_{{\sf CMP}_{a,b}}(x)\lesssim T_{\mu}(x)\lesssim T_{{\sf CMP}_{a',b'}}(x)$ for some $a,a',b,b'>0$, \emph{a geometric tail} if $\exp(-ax)\lesssim T_{\mu}(x)\lesssim \exp(-a'x)$ for some $a,a'>0$, and \emph{a Zeta-like tail} if $T_{\mu}(x)\gtrsim x^{-a}$ for some $a>0$ (see Section \ref{sec:notation} for definition of $\lesssim$). Hence, $\mu$ has a super-exponential light tail if it has a CMP-like tail, an exponential tail if it has a geometric tail, and a sub-exponential heavy tail if it has a Zeta-like tail.

The following theorem shows that certain generic property exists regarding tails of stationary distributions and QSDs for SRNs.

\begin{theorem}\label{th-tail}
  Let $\cR$ be an SRN taken with mass-action kinetics and assume $\rm{(\mathbf{H2})}$-$\rm{(\mathbf{H4})}$. Then for  $c\in\N^d_0$,
   
   (i) if $R_+<R_-$, then
    \begin{enumerate}
     \item[(i-a)] the stationary distribution on every PIC of $\sL_c$ has a CMP-like tail, and
    \item[(i-b)] the QSD on every QIC of $\sL_c$ has a CMP-like  tail, if $R>1$.
    \end{enumerate}
 
 (ii) if $R_+=R_-$ and $\alpha(c)<0$, then
    \begin{enumerate}
    \item[(ii-a)] the stationary distribution on every PIC of $\sL_c$ has a geometric tail, and
    \item[(ii-b)] the QSD on every QIC of $\sL_c$ has a geometric tail, if $R>1$.
    \end{enumerate}
 
 (iii) if $\alpha(c)=0$, then
    \begin{enumerate}
    \item[(iii-a)] the stationary distribution on every PIC of $\sL_c$ has a Zeta-like  tail, and
    \item[(iii-b)] the QSD on every QIC of $\sL_c$ has a Zeta-like tail, if $R>1$.
  \end{enumerate}
\end{theorem}

\begin{proof}
By Theorem~\ref{th-Threshold},  the conclusions directly follow by applying \cite[Theorem~4.1]{XHW20c} to the underlying CTMCs with the stationary distribution supported on a PIC of $\sL_c$ and \cite[Theorem~4.4]{XHW20c} to those with the QSD supported on a QIC of $\sL_c$, respectively.
\end{proof}

\begin{corollary}\label{cor-end-tail}
  Let $\cR$ be an SRN taken with mass-action kinetics. Assume $\rm{(\mathbf{H2})}$-$\rm{(\mathbf{H4})}$, and that $\cR$ is endotactic. Then,
  \begin{enumerate}
    \item[(i)] the stationary distribution on every PIC
    has a CMP-like tail.
    \item[(ii)] the QSD on every QIC
    has a CMP-like tail, if $R>1$.
  \end{enumerate}
 In particular,
  \begin{enumerate}
\item[(iii)] if $\cR$ is strongly endotactic, then
\begin{enumerate}
\item[(iii-a)] the stationary distribution on every PIC has a CMP-like tail, and
\item[(iii-b)] the QSD on every QIC has a CMP-like tail, provided $R>1$.
\end{enumerate}
\item[(iv)] if $\cR$ is weakly reversible, then the stationary distribution on every PIC has a CMP-like tail.
 \end{enumerate}
\end{corollary}

\begin{proof}
From Theorem~\ref{th-endotactic}, we have $R_+<R_-$, and the conclusions follow directly from Theorem~\ref{th-tail} together with Theorem~\ref{th-endotactic} and Corollary~\ref{cor-posrec}.
\end{proof}

\begin{example}
 Consider  the following 3-cycle weakly reversible SRN $\cR$:
 \[\tS\ce{->[\kappa_1]}2\tS\ce{->[\kappa_2]}3\tS\ce{->[\kappa_3]}\tS.\]
It is easy to verify that $\sP=\N$ consists of one PIC. By Corollaries~\ref{cor-posrec} and \ref{cor-end-tail}(iv), $\cR$ has a unique ergodic stationary distribution on $\N$ with a CMP tail. By direct calculation, $\cR$ is complex balanced if and only if $\kappa_1\kappa_3=\kappa_2^2$, in which case, the stationary distribution is Poisson  with parameter $\kappa_2/\kappa_3$.
\end{example}

Finally, we give an example  with  geometric tail or one with Zeta-like tail.

\begin{example}
Recall Example~\ref{ex-4}(i). The original SRN has an ergodic stationary distribution with a geometric tail while in contrast the modified SRN has an ergodic stationary distribution with a Zeta-like tail for all $\kappa\ge1$, due to the product formula for stationary distributions of BDPs.
\end{example}

\section*{Acknowledgements}
Both authors thank Dr.\! Panqiu Xia and two anonymous referees for pointing out two mistakes in an earlier version of this manuscript. The work presented in this article is supported by Novo Nordisk Foundation,
grant NNF19OC0058354. CX acknowledges the TUM Foundation Fellowship and the Alexander von Humboldt Foundation Fellowship.
\appendix

\section{Endotatic networks}\label{appendix_endotactic_network}

{Here we adopt the definition of endotatic networks from \cite{GMS14}, which as pointed out therein is equivalent to the one introduced in \cite{CNP13}.}
{\begin{definition}\label{def-max}
Let $w\in\R^d$. 
  \begin{enumerate}
    \item[(1).] The vector $w\in\R^d$ defines a preorder on $\R^d$, 
    denoted by $\le_w$:
    \[y\le_{w}y'\quad \Leftrightarrow\quad \la y,w\ra\le\la y',w\ra.\] In particular, we write $y<_wy'$ if $\la y,w\ra<\la y',w\ra$.
    \item[(2).] For a finite subset $A\subseteq\R^{\cS}$, the set of all $\le_w$-maximal elements of $A$ consists of all $x\in A$ such that
        \[x\ge_w y,\quad \text{for all}\ y\in A,\]
\item[(3).] The set $\cR_{w}\subseteq\cR$ of \emph{$w$-essential reactions} consists of all reactions whose reaction vectors are not perpendicular to $w$:
    \[\cR_w=\{y\to y'\in\cR\colon \la w,y'-y\ra\neq0\},\]
\item[(4).] The set of $\le_w$-maximal elements of $\{y\colon y\to y'\in \cR_{w}\}$ is the $w$-support of $\cR$, denoted by $\supp_w\cR$.
\end{enumerate}
\end{definition}
\begin{rem}
  $\cR_{\om}\neq\varnothing$ if and only if $\om\notin\cS^{\perp}$.
\end{rem}}
We provide one geometric interpretation of    endotatic reaction network  \cite{GMS14}. Given any vector  $w\notin\cS^{\perp}$, project the reaction graph onto the line generated by $w$, one obtains a one-dimensional reaction network. One geometric desirable feature for such endotatic reaction network is that   endotacticity is preserved under the projection. Hence, we require a reaction network is endotatic if and only if its projection to any line generated by $w\notin\cS^{\perp}$ is endotatic. A second desirable feature for endotatic reaction network with \emph{deterministic} mass-action kinetics is  ``dissipativity''. It  is anticipated that mass-action endotatic reaction networks are permanent in the sense that the dynamical system admit a compact positively invariant subset, which usually refers to the \emph{Global Attractivity Conjecture}.

We   introduce the formal definition of endotactic networks as well as strongly endotatic reaction networks, based on the $w$-maximal elements and subsets of a reaction network aforementioned.

 \begin{definition}\label{def-endotactic}
(1) A reaction network $\cR$ is \emph{$w$-endotactic} for some $w\in\R^{d}$ if $$y'<_w y$$ for all $w$-essential reactions $y\to y'$ with $y\in\supp_w\cR$. In particular, $\cR$ is \emph{endotactic} if it is $w$-endotatic for all $w\in\R^d$.

(2) A reaction network $\cR$ is \emph{strongly endotactic} if it is endotactic and for every $w\notin\sS^{\perp}$, there exists a reaction $y\to y'\in\cR$ such that
    \begin{enumerate}
      \item[(2-a)] $y'<_wy$ and
      \item[(2-b)] $y$ is $\le_w$-maximal among all reactants: For every reactant $x$ of $\cR$, $x\le_w y$.
    \end{enumerate}
 \end{definition}
 
 A geometric intuitive equivalent definition for a strongly endotatic network is given below \cite[Remark~3.13]{GMS14}. Let $\overline{\cC}_+$ be the convex hull (called the \emph{reactant polytope}) of the set $\cC_+$ of all reactants of $\cR$. A reaction $y\to y'$ \emph{points out of a set $A$} means that the line segment from $y$ to $y'$ intersects $A$ only at the point $y$. Recall that a \emph{face} of a polytope $A$ is the intersection of $A$ with any closed halfspace whose boundary is disjoint from the interior of $A$. Hence the set of faces of a polytope $A$ includes the polytope itself and the empty set. A \emph{$\le_w$-maximal face} of $A$ is a proper face of $A$ consisting of a subset of $\le_w$-maximal elements of $A$.

\begin{proposition}\label{prop-criterion-endotactic}\cite[Remark~3.13]{GMS14}
A reaction network is strongly endotactic if and only if both of the following two conditions hold:
\begin{enumerate}
\item[(i)]  no reaction with
its reactant on the boundary of $\overline{\cC}_+$ points out of $\overline{\cC}_+$,
\item[(ii)]  for all $w\notin\sS^{\perp}$, every $\le_w$-maximal face of $\overline{\cC}_+$ contains a reactant $y$
with a reaction $y\to y'$  pointing out of the face (either along the boundary of $\overline{\cC}_+$ or into the relative interior of $\overline{\cC}_+$).
\end{enumerate}
\end{proposition}

Another geometric verification of endotaticity relies on a \emph{parallel sweep test} \cite{CNP13}.

\section{Implications among parameters}
{
\begin{proposition}\label{prop-Imp} Let $\cR$ be an SRN. Assume  mass-action kinetics,  $\rm{(\mathbf{H2})}$, $\rm{(\mathbf{H3})}$ and $\rm{(\mathbf{H5})}$. Let $\gamma_c$ be defined as in \eqref{Gamma}.
\begin{enumerate}
\item[(i)] If $\rm{(\mathbf{H4})}$ holds, then $R\ge1$.
\item[(ii)] For $c\in\N^d_0$, if $\gamma_c\le0$, then $\beta_c<0$.
\item[(iii)] If $R=0$, then $\alpha>0$.
\item[(iv)] Assume $\rm{(\mathbf{H4})}$. If $R=1$, then $d=1$. In addition, if $\alpha=0$, then $\gamma_c\ge0$ for all $c\in\N_0$, and $\gamma_c=0$ if and only if $y=1$ for all reactants $y$, in which case, $\sT=\{0\}$ and $\sP=\varnothing$.
\end{enumerate}
\end{proposition}

\begin{proof}
(i) Since $\cR_-\neq\varnothing$, there exists a reaction $y\to y'\in\cR_-$ with $\frac{y'-y}{\omega^*}\in(-\N)$. Since $\om^*\neq0$ and the first coordinate $\om^*_j$ of $\om^*$ is positive for some $1\le j\le d$, we have $0\le y_j'<y_j$, which implies the order $R\ge y_j\ge1$. 

(ii) By definition, $\beta_c=\gamma_c-\vartheta_c$, where \begin{align*}
    \vartheta_c=&\frac{1}{2}\lim_{\substack{x_1\to\infty \\ x\in\sL_c}}\frac{\sum_{y\to y'\in\cR}\kappa_{y\to y'}x^{\underline{y}}(y'_1-y_1)^2}{x_1^{R}}\\
    =&\frac{1}{2}(\om^*_1)^{-R}\sum_{\|y\|_1=R}(y'_1-y_1)^2\kappa_{y\to y'}\prod_{j=1}^d\lt(\om^*_j\rt)^{y_j}.\end{align*} The derivation of the second equality is analogous to the formula of $\alpha$ as explained in the proof of Proposition~\ref{prop-formula}. Hence $\vartheta_c$ is independent of $c$ and is positive, which implies $\beta_c<0$ provided $\gamma_c\le0$.
  
  (iii) If $R=0$, then $\cR_-=\varnothing$ and all reactants are 0, due to $\rm{(\mathbf{H3})}$. By  Proposition~\ref{prop-formula}, $$\alpha=(\om^*_1)^{-R}\sum_{\|y\|_1=R,y\to y'\in\cR_+}(y'_1-y_1)\kappa_{y\to y'}\prod_{l=1}^d\lt(\om^*_j\rt)^{y_j}>0.$$
\item[(iv)] By $\rm{(\mathbf{H3})}$, $R=1$ implies all reactants have one coordinate being 1 and the rest being 0. By $\rm{(\mathbf{H4})}$-$\rm{(\mathbf{H5})}$, there exists a reaction $y\to0$ since $\cR_-\neq\varnothing$. Moreover, $\frac{y}{\om^*}\in\N$. Hence $\om^*=y$. By $\rm{(\mathbf{H5})}$ again, $y$ has no zero coordinate. This shows that $d=1$. In addition, assume $\alpha=0$. By \eqref{Gamma}, \begin{align*}\gamma_c=&\lim_{\substack{x_1\to\infty \\ x\in\sL_c}}\sum_{y=0}^1\kappa_{y\to y'}x^{\underline{y}}(y'-y)\\
    =&\sum_{y=0}\kappa_{y\to y'}(y'-y)+x\sum_{y=0}^1\kappa_{y\to y'}(y'-y)\\
    =&\sum_{y=0}\kappa_{y\to y'}(y'-y)+x\alpha\\
    =&\sum_{y=0}\kappa_{y\to y'}(y'-y)\ge0.
    \end{align*}
    Hence $\gamma_c=0$ if and only if $\{y\to y'\in\cR\colon y=0\}=\varnothing$, in other words, all reactants $y=1$ since $R=1$. In this case, due to the fact that $\alpha=0$, we have $\cR_-=\{\tS_1\ce{->[]}0\}$ which implies that $\sT=\{0\}$. Since $-1\in\om^*\Z$, we have $\om^*=1$. By 
$\cR_+\neq\varnothing$, we have all states in $\N$ lead to $o$ and $\sP=\varnothing$.
\end{proof}
\bibliographystyle{plain}
\bibliography{references}

\end{document}